\documentclass[11pt]{amsart}

\usepackage{amscd}
\usepackage{amsmath, amssymb}
\usepackage{amsfonts}
\newcommand{\de}{\partial}

\newcommand{\vol}{\mathrm{Vol}}

\newcommand{\ve}{\varepsilon}

\renewcommand{\leq}{\leqslant}
\renewcommand{\geq}{\geqslant}

\numberwithin{equation}{section}

\begin{document}
\newtheorem{claim}{Claim}
\newtheorem{theorem}{Theorem}[section]
\newtheorem{lemma}[theorem]{Lemma}
\newtheorem{corollary}[theorem]{Corollary}
\newtheorem{proposition}[theorem]{Proposition}
\newtheorem{question}{question}[section]
\newtheorem{conjecture}[theorem]{Conjecture}

\theoremstyle{definition}
\newtheorem{remark}[theorem]{Remark}
\author{Valentino Tosatti}
\title[Zariski decompositions]{Orthogonality of divisorial Zariski decompositions for classes with volume zero}
\address{Department of Mathematics, Northwestern University, 2033 Sheridan Road, Evanston, IL 60208}
\email{tosatti@math.northwestern.edu}
\thanks{Supported in part by a Sloan Research Fellowship and NSF grant DMS-1308988. I am grateful to John Lesieutre for useful discussions and to Xiaokui Yang for helpful comments.}

\begin{abstract} We show that the orthogonality conjecture for divisorial Zariski decompositions on compact K\"ahler manifolds
holds for pseudoeffective $(1,1)$ classes with volume zero.
\end{abstract}
\maketitle

\section{Introduction}
The orthogonality conjecture of divisorial Zariski decompositions \cite{BDPP} states the following:
\begin{conjecture}\label{orto}
Let $(X^n,\omega)$ be a compact K\"ahler manifold, and $\alpha$ a pseudoeffective $(1,1)$ class. Then
\begin{equation}\label{ort}
\langle\alpha^{n-1}\rangle\cdot\alpha=\vol(\alpha).
\end{equation}
\end{conjecture}
Here $\vol(\alpha)$ denotes the volume of the class $\alpha$ \cite{Bo} and $\langle \cdot\rangle$ is the moving intersection product of classes as introduced by Boucksom \cite{BoT,BDPP,BEGZ,BFJ}.
The name of this problem comes from the following observation. If we choose an approximate Zariski decomposition of the class $\alpha$ (see section \ref{sectmain} for the precise details of the construction),
given by suitable modifications $\mu_\delta:X_\delta\to X$, for all small $\delta>0$, with $\mu_\delta^*(\alpha+\delta\omega)=\theta_\delta+[E_\delta]$ with $\theta_\delta$ a semipositive class and $E_\delta$ an effective $\mathbb{R}$-divisor, then \eqref{ort} is equivalent to
$$\lim_{\delta\downarrow 0}\int_{X_\delta}\theta_\delta^{n-1}\wedge[E_\delta]=0,$$
which explains the name. This conjecture was first raised by Nakamaye \cite[p.566]{Nak2} when $X$ is projective and $\alpha=c_1(L)$ for some line bundle $L$. This was solved by 
Boucksom-Demailly-P\u{a}un-Peternell \cite{BDPP} who more generally proved Conjecture \ref{orto} when $X$ is projective and $\alpha$ belongs to the real N\'eron-Severi group, and posed it in the general case.
It was observed in \cite{BDPP, BFJ} that Conjecture \ref{orto} is equivalent to several other powerful statements, including the fact that the dual cone of the pseudoeffective cone $\mathcal{E}$ is the movable cone $\mathcal{M}$, the weak transcendental Morse inequalities, and the $C^1$ differentiability of the volume function on the big cone. Very recently, Witt-Nystr\"om \cite{WN} has proved Conjecture \ref{orto} when $X$ is projective.

Our main result is the following:
\begin{theorem}\label{zero}
Conjecture \ref{orto} holds if $\vol(\alpha)=0$.
\end{theorem}
The strategy of proof is similar to the one in \cite{BDPP} with one crucial difference. Since the weak transcendental Morse inequality
\begin{equation}\label{goal}
\vol(\alpha-\beta)\geq \int_X\alpha^n-n\int_X\alpha^{n-1}\wedge\beta,
\end{equation}
for the difference of two nef classes $\alpha,\beta$ remains conjectural, we employ instead the weaker version
\begin{equation}\label{goal2}
\vol(\alpha-\beta)\geq \frac{\left(\int_X\alpha^n-n\int_X\alpha^{n-1}\wedge\beta\right)^n}{\left(\int_X\alpha^n\right)^{n-1}},
\end{equation}
which was proved independently in \cite{To,Po2}, using the mass concentration technique of Demailly-P\u{a}un \cite{DP} and its recent improvements by Chiose \cite{Ch}, Xiao \cite{Xi} and Popovici \cite{Po}. Inequality \eqref{goal2} is too weak to prove Conjecture \ref{orto} in general, but the fact that the numerator on the RHS has the correct form turns out to be enough to prove Theorem \ref{zero}.

If we define the ``difference function''
$\mathcal{D}:\mathcal{E}\to\mathbb{R}$ by 
\begin{equation}\label{d}
\mathcal{D}(\alpha):=\langle\alpha^{n-1}\rangle\cdot\alpha-\vol(\alpha),
\end{equation}
then Conjecture \ref{orto} simply states that $\mathcal{D}$ vanishes identically on $\mathcal{E}$. As a corollary of Theorem \ref{zero} we have:

\begin{corollary}\label{cor}
The function $\mathcal{D}:\mathcal{E}\to\mathbb{R}$ is nonnegative, continuous on $\mathcal{E}$, and vanishes on its boundary. 
\end{corollary}

Theorem \ref{zero} and Corollary \ref{cor} are proved in section \ref{sectmain}. In section \ref{sectrem} we will make some further remarks on the function $\mathcal{D}$, and on the relation between Theorem \ref{zero} and the ``cone duality'' conjecture.

\section{The main theorem}\label{sectmain}
In this section we give the proof of Theorem \ref{zero} and Corollary \ref{cor}.

\begin{proof}[Proof of Theorem \ref{zero}] Let $\alpha$ be any pseudoeffective $(1,1)$ class on $X$.
By definition of moving intersection products, for any $1\leq p\leq n$ the real $(p,p)$ cohomology class $\langle\alpha^{p}\rangle$ is defined to be
$$\langle\alpha^{p}\rangle=\lim_{\delta\downarrow 0}\langle (\alpha+\delta\omega)^{p}\rangle,$$
where for $\delta>0$ the class $\alpha+\delta\omega$ is big, and in this case we define
$$\langle(\alpha+\delta\omega)^{p}\rangle=[\langle T_{min,\delta}^{p}\rangle],$$
where $T_{min,\delta}$ is any positive current with minimal singularities in the class $\alpha+\delta\omega$, and $\langle T_{min,\delta}^{p}\rangle$ denotes the non-pluripolar product \cite{BEGZ}.

The volume of $\alpha$ (see \cite{Bo}) is in fact equal to the moving self-intersection product 
$$\vol(\alpha)=\langle\alpha^n\rangle=\lim_{\delta\downarrow 0}\langle(\alpha+\delta\omega)^n\rangle.$$

We now review the well-known construction of approximate Zariski decompositions \cite{BoT,BDPP, BEGZ,BFJ}, following roughly the argument in \cite[Proposition 1.18]{BEGZ}.
Applying Demailly's regularization \cite{De} to $T_{min,\frac{\delta}{2}}$ we obtain a sequence of currents $T_{\delta,\ve}, \ve>0,$ in the big class $\alpha+\frac{\delta}{2}\omega$, with analytic singularities, with $T_{\delta,\ve}\geq -\ve\omega$, and with their potentials decreasing to that of $T_{min,\frac{\delta}{2}}$ as $\ve\to 0$.
As long as $\ve\leq\frac{\delta}{2}$, we have that $T_{\delta,\ve}+\frac{\delta}{2}\omega$ and $T_{min,\frac{\delta}{2}}+\frac{\delta}{2}\omega$ are closed positive currents in the class $\alpha+\delta\omega$ whose potentials are locally bounded away from the proper analytic subvariety $A:=E_{nK}(\alpha+\frac{\delta}{2}\omega)$, and so by weak continuity of the Bedford-Taylor Monge-Amp\`ere operator along decreasing sequences we have
$$\left(T_{\delta,\ve}+\frac{\delta}{2}\omega\right)^{p}\to \left(T_{min,\frac{\delta}{2}}+\frac{\delta}{2}\omega\right)^{p},$$
weakly on $X\backslash A$ as $\ve\to 0$, for all $1\leq p\leq n$. It follows that
\[\begin{split}
\int_{X\backslash A}\left(T_{min,\frac{\delta}{2}}+\frac{\delta}{2}\omega\right)^{p}\wedge\omega^{n-p}&\leq\liminf_{\ve\to 0}\int_{X\backslash A}\left(T_{\delta,\ve}+\frac{\delta}{2}\omega\right)^{p}\wedge\omega^{n-p}\\
&\leq \int_{X\backslash A}(T_{min,\delta})^{p}\wedge\omega^{n-p},
\end{split}\]
where the last inequality follows from \cite[Theorem 1.16]{BEGZ}, since all the currents involved have small unbounded locus.
But as $\delta\to 0$ both the LHS and the RHS converge to
$$\int_X\langle \alpha^p\rangle\wedge\omega^{n-p},$$
and so we may choose a sequence $\ve(\delta)\to 0$ such that the currents $T_\delta:=T_{\delta,\ve(\delta)}-\frac{\delta}{2}\omega$ in the class $\alpha$ have analytic singularities,
satisfy $T_\delta\geq -\delta\omega$, and are such that
$$\lim_{\delta\downarrow 0}\langle (T_{\delta}+\delta\omega)^p\rangle=\langle\alpha^p\rangle,$$
for all $1\leq p\leq n$. 
Let $\mu_\delta:X_\delta\to X$ be a resolution of the singularities of $T_\delta+\delta\omega$, so that
$$\mu_\delta^*(T_\delta+\delta\omega)=\theta_\delta+[E_\delta],$$
so that $\theta_\delta$ is a smooth semipositive form, $E_\delta$ is an effective $\mathbb{R}$-divisor and $[E_\delta]$ denotes the current of integration. We refer to this construction as an approximate Zariski decomposition for the class $\alpha$.

As discussed above, we have
$$\vol(\alpha)=\lim_{\delta\downarrow 0}\int_X \langle (T_\delta+\delta\omega)^{n}\rangle=\lim_{\delta\downarrow 0}\int_{X_\delta}\langle (\mu_\delta^*(T_\delta+\delta\omega))^{n}\rangle=\lim_{\delta\downarrow 0}\int_{X_\delta}\theta_\delta^n,$$
\[\begin{split}
\langle\alpha^{n-1}\rangle\cdot \alpha&=\lim_{\delta\downarrow 0}\int_X \langle (T_\delta+\delta\omega)^{n-1}\rangle\wedge\alpha
=\lim_{\delta\downarrow 0}\int_X \langle (T_\delta+\delta\omega)^{n-1}\rangle\wedge(\alpha+\delta\omega)\\
&=\lim_{\delta\downarrow 0}\int_{X_\delta}\langle (\mu_\delta^*(T_\delta+\delta\omega))^{n-1}\rangle\wedge\mu_\delta^*(\alpha+\delta\omega)\\
&
=\lim_{\delta\downarrow 0}\int_{X_\delta}(\theta_\delta^n+\theta_\delta^{n-1}\wedge[E_\delta]),
\end{split}\]
so that the orthogonality relation \eqref{ort} is in general equivalent to the statement that
\begin{equation}\label{ort2}
\lim_{\delta\downarrow 0}\int_{X_\delta}\theta_\delta^{n-1}\wedge[E_\delta]=0.
\end{equation}
We now follow \cite{BDPP}, and fix a constant $C_0$ such that $C_0\omega\pm(\alpha+\delta\omega)$ is nef. We write
$$E_\delta=\mu_\delta^*(\alpha+\delta\omega+C_0\omega)-(\theta_\delta+C_0\mu_\delta^*\omega),$$
as the difference of two nef classes. For $t\in [0,1]$ write
$$\theta_\delta+tE_\delta=A-B,$$
where
$$A=\theta_\delta+t\mu_\delta^*(\alpha+\delta\omega+C_0\omega),$$
$$B=t(\theta_\delta+C_0\mu_\delta^*\omega),$$
and $A,B$ are nef. Then
$$\vol(\alpha+\delta\omega)=\vol(\theta_\delta+E_\delta)\geq \vol(\theta_\delta+tE_\delta)=\vol(A-B).$$
We use \cite[Theorem 1.1]{To} (also independently obtained in \cite[Theorem 3.5]{Po2}) and obtain
\begin{equation}\label{key}
\vol(A-B)\geq \frac{\left(\int_{X_\delta} A^n -n\int_{X_\delta} A^{n-1}\wedge B\right)^n}{\left(\int_{X_\delta}A^n\right)^{n-1}}.
\end{equation}
We follow the same argument as in \cite[p.218]{BDPP} (see also \cite[(11.15)]{Laz}) and estimate
$$\int_{X_\delta} A^n -n\int_{X_\delta} A^{n-1}\wedge B\geq \int_{X_\delta}\theta_\delta^n+nt\int_{X_\delta}\theta_\delta^{n-1}\wedge [E_\delta] - 5n^2t^2C_0^n\int_X\omega^n,$$
as long as $t\leq\frac{1}{10n}$. We choose
$$t=\frac{\int_{X_\delta}\theta_\delta^{n-1}\wedge [E_\delta]}{10n C_0^n\int_X\omega^n},$$
which is easily seen to be less than $\frac{1}{10n}$, and 
so we obtain
$$\int_{X_\delta} A^n -n\int_{X_\delta} A^{n-1}\wedge B\geq \int_{X_\delta}\theta_\delta^n+\frac{1}{20}\frac{\left(\int_{X_\delta}\theta_\delta^{n-1}\wedge [E_\delta]\right)^2}{C_0^n\int_X\omega^n},$$
and plugging this into \eqref{key} we obtain
$$\vol(\alpha+\delta\omega)\geq \frac{\left(\int_{X_\delta}\theta_\delta^n+\frac{1}{20}\frac{\left(\int_{X_\delta}\theta_\delta^{n-1}\wedge [E_\delta]\right)^2}{C_0^n\int_X\omega^n}\right)^n}{\left(\int_{X_\delta}A^n\right)^{n-1}}.$$
We also have
\begin{equation}\label{bad}
\int_{X_\delta}A^n=\int_{X_\delta}(\theta_\delta+t\mu_\delta^*(\alpha+\delta\omega+C_0\omega))^n\leq \int_{X_\delta}\theta_\delta^n + Ct\leq C,
\end{equation}
and so
$$\int_{X_\delta}\theta_\delta^n+\frac{1}{20}\frac{\left(\int_{X_\delta}\theta_\delta^{n-1}\wedge [E_\delta]\right)^2}{C_0^n\int_X\omega^n}\leq C\vol(\alpha+\delta\omega)^{\frac{1}{n}}\to C\vol(\alpha)^{\frac{1}{n}}=0,$$
as $\delta\to 0$, and we conclude that
$$\int_{X_\delta}\theta_\delta^{n-1}\wedge [E_\delta]\to 0,$$
which proves \eqref{ort2}.
\end{proof}

\begin{remark}
Some of the estimates in this proof, such as for example \eqref{bad}, are far from being sharp. It is easy to make them sharp, but this does not appear to give any useful improvement.
\end{remark}

\begin{proof}[Proof of Corollary \ref{cor}]
Using the notation as in the proof of Theorem \ref{zero}, we have that
$$\mathcal{D}(\alpha)=\lim_{\delta\downarrow 0}\int_{X_\delta}\theta_\delta^{n-1}\wedge[E_\delta]\geq 0.$$ 
It follows easily from the definitions that moving intersection products are upper-semicontinuous on the pseudoeffective cone and continuous in its interior (the big cone), while it was proved in \cite{Bo} that the volume function is continuous on the whole pseudoeffective cone. It follows that $\mathcal{D}$ is continuous on the big cone and upper-semicontinuous on the pseudoeffective cone. By Theorem \ref{zero}, $\mathcal{D}$ vanishes on its boundary, and hence it is continuous on all of $\mathcal{E}$.
\end{proof}

\section{Further Remarks}\label{sectrem}
In this section we collect some further remarks on the function $\mathcal{D}$ and on the cone duality conjecture of \cite{BDPP}.

The function $\mathcal{D}$ defined in \eqref{d} clearly vanishes on the nef cone, where moving intersection products equal usual intersection products (see e.g. \cite{BEGZ}). More generally we have the following result, which is a simple consequence of the author's work with Collins \cite{CT} and was also observed by Deng \cite{Den}:

\begin{proposition}
Let $\alpha$ be a pseudoeffective $(1,1)$ class, and write $\alpha=P+N$ for its divisorial Zariski decomposition. If the class $P$ is nef then $\mathcal{D}(\alpha)=0$.
\end{proposition}
Recall here that the positive part is given by $P=\langle \alpha\rangle$, which is in general only nef in codimension $1$ \cite{Bo2, Na}.
\begin{proof}
First, we show that 
$$\langle \alpha^{n-1}\rangle\cdot N=0.$$
Since moving products are unchanged if we replace a class by its positive part (see \cite{BDPP, BFJ}), this is equivalent to showing that
$$\langle P^{n-1}\rangle\cdot N=0,$$
and since by assumption $P$ is nef, this is also equivalent to showing that
$$P^{n-1}\cdot N=0.$$
But by definition the irreducible components of $N$ are contained in $E_{nK}(\alpha)=E_{nK}(P)$ (see e.g. \cite[Claim 4.7]{Ma}), and so are also irreducible components of $E_{nK}(P)$. By the main theorem of \cite{CT} we have therefore $P^{n-1}\cdot N=0$. On the other hand since $P$ is nef we also have
$$\langle \alpha^{n-1}\rangle\cdot\alpha=\langle \alpha^{n-1}\rangle\cdot P=\langle P^{n-1}\rangle\cdot P=P^n=\vol(P)=\vol(\alpha),$$
as claimed.
\end{proof}
\begin{remark}
In fact, it is not hard to see (cf. \cite{BDPP}) that in general \eqref{ort} is equivalent to the following two relations both holding
\begin{equation}\label{orto1}
\langle\alpha^{n-1}\rangle\cdot P=\vol(\alpha),
\end{equation}
\begin{equation}\label{orto2}
\langle\alpha^{n-1}\rangle\cdot N=0.
\end{equation}
\end{remark}

It was proved in \cite{BDPP} that Conjecture \ref{orto} implies (and is in fact equivalent to) the ``cone duality'' conjecture, which states that the dual cone of the pseudoeffective cone $\mathcal{E}$ of a compact K\"ahler manifold equals the movable cone $\mathcal{M}\subset H^{n-1,n-1}(X,\mathbb{R})$. It is easy to see that $\mathcal{E}\subset\mathcal{M}^\vee$, so the point is to prove the reverse inclusion. The proof given there starts by assuming that there is a class $\alpha\in \de\mathcal{E}\cap(\mathcal{M}^\vee)^\circ$, and derives a contradiction, assuming that $X$ is projective and that the class belongs to the real N\'eron-Severi group.

In general we have the following:
\begin{proposition}
Let $\alpha\in \de\mathcal{E}\cap(\mathcal{M}^\vee)^\circ$. Then we must have $\langle\alpha^{n-1}\rangle=0$.
\end{proposition}
\begin{proof}
Assume for a contradiction that $\langle\alpha^{n-1}\rangle\neq 0$ in $H^{n-1,n-1}(X,\mathbb{R})$. Since this class is represented by a closed nonnegative $(n-1,n-1)$ current, it follows that
$$\int_X\langle\alpha^{n-1}\rangle\wedge\omega>0,$$
where $\omega$ is a fixed K\"ahler metric.
We choose an approximate Zariski decomposition $T_\delta$ with resolutions $\mu_\delta:X_\delta\to X$ as before, so that we have also
$$0<2\eta\leq\int_X\langle\alpha^{n-1}\rangle\wedge\omega=\lim_{\delta\downarrow 0} \int_{X_\delta}\theta_\delta^{n-1}\wedge \mu_\delta^*\omega=\lim_{\delta\downarrow 0} \int_X (\mu_\delta)_*(\theta_\delta^{n-1})\wedge\omega,$$
for some fixed $\eta>0$. Up to modifying the classes $\theta_\delta$ (and $[E_\delta]$) by subtracting a small multiple of $\mathrm{Exc}(\mu_\delta)$, we may also assume that the classes $\theta_\delta$ are K\"ahler on $X_\delta$, so that the pushforwards $(\mu_\delta)_*(\theta_\delta^{n-1})$ are movable classes on $X$.


Let $\ve>0$ be such that $\alpha-\ve\omega\in\mathcal{M}^\vee$. Integrating $\alpha-\ve \omega$ against the class $(\mu_\delta)_*(\theta_\delta^{n-1})\in\mathcal{M}$ we obtain
$$\int_X\alpha\wedge (\mu_\delta)_*(\theta_\delta^{n-1})\geq\ve\int_X\omega\wedge (\mu_\delta)_*(\theta_\delta^{n-1})\geq \ve\eta>0,$$
for all $\delta>0$ small. But we also have
$$\int_X\alpha\wedge (\mu_\delta)_*(\theta_\delta^{n-1})=\int_{X_\delta}\mu_\delta^*\alpha\wedge \theta_\delta^{n-1}\leq\int_{X_\delta}
\mu_\delta^*(\alpha+\delta\omega)\wedge \theta_\delta^{n-1}=\int_{X_\delta}(\theta_\delta^n+\theta_\delta^{n-1}\wedge[E_\delta]),$$
and putting these together we have
$$\int_{X_\delta}(\theta_\delta^n+\theta_\delta^{n-1}\wedge[E_\delta])\geq \ve\eta,$$
for all $\delta>0$ small. Since $\vol(\alpha)=0$, this implies
$$\int_{X_\delta}\theta_\delta^{n-1}\wedge[E_\delta]\geq \frac{\ve\eta}{2},$$
for all $\delta>0$ small, which is a contradiction to Theorem \ref{zero}. Therefore we must have that $\langle\alpha^{n-1}\rangle=0$.
\end{proof}
\begin{remark}
Even though the class $\alpha\in \de\mathcal{E}\cap(\mathcal{M}^\vee)^\circ$ does satisfy the orthogonality conjecture (say as in \eqref{ort2}) by Theorem \ref{zero}, because it has volume zero, this is not enough to derive a contradiction in general. Indeed, to make the argument in \cite{BDPP} go through, one would need the following quantitative version of orthogonality
$$\left(\int_{X_\delta}\theta_\delta^{n-1}\wedge[E_\delta]\right)^2\leq C\left(\vol(\alpha+\delta\omega)-\int_{X_\delta}\theta_\delta^n\right),$$
(cf. \cite[Theorem 4.1]{BDPP}) which does not follow from the arguments of Theorem \ref{zero}.
\end{remark}

\end{document}